\documentclass[12pt,reqno]{amsart}
\usepackage{pdfsync}
\usepackage{xcolor}
\usepackage[normalem]{ulem}
\usepackage{amsmath, amscd}
\usepackage{amsfonts, amssymb}
\usepackage{mathtools}
\usepackage{xcolor}
\usepackage{version}
\usepackage{mathrsfs}
\usepackage{nicefrac}
\usepackage[all]{xy}
\usepackage[all]{xy}

\numberwithin{equation}{section}

\def\1#1{\overline{#1}}
\def\2#1{\widetilde{#1}}
\def\3#1{\widehat{#1}}
\def\4#1{\mathbb{#1}}
\def\5#1{\frak{#1}}
\def\6#1{{\mathcal{#1}}}

\newcommand{\R}{\mathbb R}

\newcommand{\C}{\mathbb C}

\renewcommand{\H}{\mathbb H}

\newcommand{\B}{\mathbb B}

\newcommand{\Aut}{{\sf Aut}}
\newcommand{\D}{\mathbb D}

\newcommand{\Ha}{\mathbb H}

\def\v{\varphi}

\def\Im{{\sf Im}\,}

\def\SD#1{c(#1)}

\newtheorem{theorem}{Theorem}[section]

\newtheorem{proposition}[theorem]{Proposition}
\newtheorem{corollary}[theorem]{Corollary}

\theoremstyle{definition}
\newtheorem{definition}[theorem]{Definition}

\theoremstyle{remark}
\newtheorem{remark}[theorem]{Remark}
\numberwithin{equation}{section}

\def\id{{\sf id}}
\newcommand{\Real}{\mathbb{R}}
\newcommand{\Natural}{\mathbb{N}}

\newcommand{\mcite}[1]{\csname b@#1\endcsname}

\long\def\REM#1{\relax}

\newcommand{\Maponto}
{\xrightarrow{\hbox{\lower.2ex\hbox{$\scriptstyle \smash{\mathsf{onto}}$}}\,}}
\newcommand{\Mapinto}
{\xrightarrow{\hbox{\lower.2ex\hbox{$\scriptstyle \smash{\mathsf{into}}$}}\,}}

\newcommand{\HIDECORRECTIONS}{%
\newcommand{\nv}[1]{##1}
\newcommand{\dv}[1]{\relax}%
}%

\title{Valiron and Abel equations for holomorphic self-maps of the polydisc}
\author[L. Arosio]{Leandro Arosio$^\ast$}
\address{Dipartimento Di Matematica\\
Universit\`{a} di Roma \textquotedblleft Tor Vergata\textquotedblright\ \\
Via Della Ricerca Scientifica 1, 00133 \\
Roma, Italy} \email{arosio@mat.uniroma2.it}

\author[P. Gumenyuk]{Pavel Gumenyuk$^\dag$}
\address{Department of Mathematics and Natural Sciences,
University of Stavanger\\
4036 Stavanger, Norway} \email{pavel.gumenyuk@uis.no}

\thanks{$^{*}$\,Supported by the ERC grant ``HEVO - Holomorphic Evolution Equations'' (No.\, 277691), and by the SIR grant ``NEWHOLITE - New methods in holomorphic iteration'' (No.\, RBSI14CFME)}
\thanks{$^\dag$\,Partially supported by the FIRB program ``Futuro in Ricerca 2008'', project \textit{Geometria Differenziale Complessa e Dinamica Olomorfa} (No.\,RBFR08B2HY)}

\begin{document}
\HIDECORRECTIONS
\begin{abstract}
We introduce a notion of hyperbolicity and parabolicity for a holomorphic self-map $f\colon \Delta^N\to \Delta^N$ of the polydisc which does not admit fixed points in $\Delta^N$.
We generalize to the polydisc two classical one-variable results: we solve the Valiron equation for a hyperbolic $f$ and the Abel equation for a parabolic nonzero-step $f$.
This is done by studying the canonical Kobayashi hyperbolic semi-model of $f$ and by obtaining a normal form for the automorphisms of the polydisc.
In the case of the Valiron equation we also describe the space of all solutions.
\end{abstract}
\maketitle
\tableofcontents


\section{Introduction}
A holomorphic self-map $f\colon \Delta\to \Delta$ of the unit disc can be classified in three types according to its  dynamical behavior.
The self-map $f$ is called {\sl elliptic} if it admits a fixed point $p\in \Delta$.
If $f$ is not an automorphism of $\Delta$, i.e. if $|f'(p)|<1$, then $(f^n)$ converges to the constant map $p$ uniformly on compact subsets.
If $f$ admits no fixed points in $\Delta$, then the classical Denjoy--Wolff theorem states that
there exists a point $p\in \partial \D$ such that  $(f^n)$ converges to the constant map $p$ uniformly on compact subsets, and such that
$$\angle \lim_{z\to p}f(z)=p \quad\mbox{and}\quad\lambda_{f}\coloneqq \liminf_{z\to p} \frac{1-|f(z)|}{1-|z|}\leq1,$$ where $\angle \lim$ stands for the non-tangential limits.
The point $p$ is called the {\sl Denjoy--Wolff point} of $f$ and the number $\lambda_{f}$ is called the {\sl dilation of $f$} at~$p$.
The self-map $f$ is called {\sl parabolic}  if $\lambda_f=1$ and
 {\sl hyperbolic}  if $\lambda_f<1$.

By $k_\Delta$ denote the Poincar\'e distance on $\Delta$.
The class of all parabolic self-maps splits into two sub-classes: parabolic self-maps $f$ {\sl of non-zero step}, for which the limit of the non-increasing sequence $\big(k_\Delta(f^n(z),f^{n+1}(z))\big)_{n\ge0}$ is positive for some (and hence for any) $z\in \Delta$,  and  those {\sl of zero step}, for which this limit is identically zero.

Each type in this classification is associated with a functional equation that gives a ``model'' for the dynamics of the self-map.
In the following theorem we summarize the results obtained by K\"onigs \cite{Ko}, Valiron \cite{Va}, Pommerenke \cite{Po}, and by Baker and Pommerenke \cite{BaPo}.
It is also worth to mention the important work of Cowen \cite{cowen}, who unified these equations in a common framework. We follow Cowen's presentation of these results.
We denote by $\H$ the upper-half plane.

\begin{theorem}\label{general}
Let $f\colon \Delta\to \Delta$ be a holomorphic self-map. \nv{If $f$ is elliptic, we suppose that $0<|f'(p)|<1$, where $p\in\Delta$ is the (unique) fixed point of~$f$.} Then there exists an $f$-invariant domain $U\subset \Delta$ such that for all $z\in \Delta$ the orbit $(f^n(z))$ eventually lies in $U$ and such that $f|_U$ is univalent. \nv{Moreover:}
\begin{enumerate}
\item[(1)] {\rm [K\"onigs]}  If $f$ is an elliptic map, then there exists a holomorphic function $\Theta\colon \Delta\to \C$ which is univalent on $U$, solves the {\sl Schr\"oder equation}
\begin{equation}\label{Seq}\Theta (f(z))=f'(p) \Theta(z)\quad \text{for all~}z\in\Delta,
\end{equation}
and satisfies $\bigcup_{n\geq 0} f'(p)^{-n}\,\Theta(\Delta)=\C$.\vskip1ex

\item[(2)] {\rm [Valiron]}  If $f$ is a hyperbolic map with dilation $\lambda_f$ at its Denjoy--Wolff point, then
 there exists a holomorphic function $\Theta\colon \Delta\to \H$ which is univalent on $U$, solves the {\sl Valiron equation}
\begin{equation}\label{Veq}\Theta (f(z))=\frac{1}{\lambda_{f}}\,\Theta(z)\quad \text{for all~}z\in\Delta,
\end{equation}
and satisfies $\bigcup_{n\geq 0} \lambda_f^{n}\,\Theta(\Delta)=\H$.\vskip1ex

\item[(3)] {\rm [Pommerenke]}  If $f$ is a parabolic map of non-zero step, then
 there exists a holomorphic function $\Theta\colon \Delta\to \H$  which is univalent on $U$, solves the {\sl Abel equation}
\begin{equation}\label{Aeq}\Theta (f(z))=\Theta(z)\pm 1\quad \text{for all~}z\in\Delta,
\end{equation}
and satisfies $\bigcup_{n\geq 0}\big(\Theta(\Delta)\mp n\big)=\H$.\vskip1ex

\item[(4)] {\rm [Baker--Pommerenke]}  If $f$ is a parabolic map of zero step, then
 there exists a holomorphic function $\Theta\colon \Delta\to \C$  which is univalent on $U$, solves the {\sl Abel equation}
\begin{equation}\label{Aeq2}\Theta (f(z))=\Theta(z)+1\quad \text{for all~}z\in\Delta,
\end{equation}
and satisfies
 $\bigcup_{n\geq 0}\big(\Theta(\Delta)-n\big)=\C$.
\end{enumerate}
\end{theorem}
\begin{remark}
It is easy to see that solutions to the Schr\"oder equation (\ref{Seq}) differ just by a constant complex factor. Analogous (but much deeper) uniqueness result for the Valiron equation (\ref{Veq}) is due to Bracci and Poggi-Corradini~\cite{BP}: every holomorphic solution $\eta\colon \D\to \H$ to the Valiron equation is a positive multiple of $\Theta$ and satisfies  ${\bigcup_{n\geq 0} \lambda_f^{n}\,\eta(\Delta)=\H}$.

Some (weaker) uniqueness results for the Abel equation are known,  see, e.g.,~\cite{zoo} for a detailed discussion. The reason for the $\pm$ sign in the Abel equation (\ref{Aeq}) is that we only consider solutions with values in the upper half-plane~$\H$. This is not the case for the Abel equation (\ref{Aeq2}), for which solutions with values in the whole~$\C$ are allowed.
\end{remark}

The Denjoy--Wolff theorem was generalized to the unit ball $\B^N$ by Herv\'e \cite{herve}, to any $C^2$-smooth bounded strongly convex  domain by Abate \cite{abate}, and to arbitrary (not necessarily smooth)  bounded  strongly convex domains by Budzy\'nska \cite{Bu}, see also \cite{AR, R1}. Using the dilation at the Denjoy--Wolff point, one can extend the dynamical classification to holomorphic self-maps on the ball. This made possible to generalize Theorem~\ref{general}  in several ways,
see, e.g., \cite{ABmod} for a brief history.

For the polydisc $\Delta^N$, holomorphic dynamics has been studied by a number of specialists, e.g., by Herv\'e~\cite{herveP}, Frosini~\cite{frosini}, Abate and Raissy~\cite{AR}.
The main issue in this case is that the Denjoy--Wolff theorem fails, as the following example from~\cite{A} shows. Let $f\colon \Delta^2\to\Delta^2$ be defined by  $f(z,w)\coloneqq (\lambda z, \frac{1+w}{3-w})$, where $ |\lambda|=1,\lambda\neq 1 $. Then $f$ has no fixed points in~$\Delta^2$, but $(f^n)$ does not converge.

Hence, in the polydisc case, a different approach is needed in order to obtain a dynamical classification.
In~\cite{ABmod} it is proved that given a holomorphic self-map $f\colon \B^N\to \B^N$ with no fixed points in $\B^N$,  one can determine whether $f$ is parabolic or hyperbolic by looking at how fast the orbits diverge to infinity w.r.t. the Kobayashi distance in the ball. The divergence rate is defined in~\cite{ABmod} not only in~$\B^N$ but on any complex manifold\nv{. This allows us to use the divergence rate} in order to define dynamical types of holomorphic self-maps of  the polydisc, see Definition~\ref{definition}.

The main results of this note are the generalizations to the polydisc of assertions~(2) and~(3) of  Theorem~\ref{general}.
\begin{theorem}\label{TH_main-Valiron}
Let $f\colon \Delta^N\to \Delta^N$ be a hyperbolic holomorphic self-map with dilation $\lambda_{f}$.
Then there exists a holomorphic function $\Theta\colon \Delta^N\to \H$ that solves the Valiron equation
$$\Theta( f(z))=\frac{1}{\lambda_{f}}\,\Theta(z)\quad \text{for all~}z\in\Delta,$$
and satisfies $\bigcup_{n\geq 0} \lambda_f^{n}\,\Theta(\Delta)=\H$.
\end{theorem}

\begin{theorem}\label{TH_main-Abel}
Let $f\colon \Delta^N\to \Delta^N$ be a parabolic holomorphic self-map, with nonzero-step.
Then there exists a holomorphic function $\Theta\colon \Delta^N\to \H$ that solves the Abel equation
$$\Theta( f(z))=\Theta(z)\pm 1 \quad \text{for all~}z\in\Delta,$$
and satisfies $\bigcup_{n\geq 0}\big(\Theta(\Delta)\mp n\big)=\H$.
\end{theorem}

The proof of both theorems consists of two main steps. First we apply the theory of canonical semi-models developed by Bracci and the first-named author in  \cite{ABmod,Amod} and a result of Heath and Suffridge~\cite{HS} to reduce the problem to the case of an automorphism $\tau$ of the polydisc.
Then we solve the problem for $\tau$ by conjugating it to a suitable normal form.

Theorem~\ref{TH_main-Abel} is a part of Theorem~\ref{par}, which we prove in Section~\ref{S_Abel}.
Theorem~\ref{TH_main-Valiron} follows, in view of Remark~\ref{RM_inf_Valiron-functions}, from Theorem~\ref{TH_Valiron-equation}, which we prove in Section~\ref{S_Valiron}.
Moreover, Theorem~\ref{TH_Valiron-equation} gives a complete description of all solutions to the Valiron equation with values in~$\H$, generalizing the uniqueness result by Bracci and Poggi-Corradini~\cite{BP} for the unit disc.

\section{Preliminaries}

For a complex  manifold $X$, we denote by $k_X$ its Kobayashi pseudo-distance.

\begin{definition}
Let $f\colon X\to X$ be a holomorphic self-map of a complex manifold~$X$.
 The {\sl step} $s(x)$ of $f$ at $x\in X$ is the limit
$$s(x)\coloneqq\lim_{n\to\infty}k_X(f^n(x), f^{n+1}(x)).$$ This limit exists because the sequence $(k_X(f^n(x), f^{n+1}(x))_{n\geq 0}$ is non-increasing.

The {\sl divergence rate} $c(f)$ of $f$ is the limit $$c(f)\coloneqq\lim_{m\to\infty}\frac{k_X(f^m(x),x)}{m},$$ which, according to~\cite{ABmod}, exists and does not depend on the choice of~$x\in X$.
\end{definition}

\begin{remark}
In the context of non-expansive mapping theory in Banach spaces and in
the Hilbert ball, $s(x)$ and $c(f)$ were considered in~\cite{BBR,R2,R3}.
\end{remark}

The following result is proved in \cite{ABmod}.
\begin{theorem}
Let $f\colon \B^q\to \B^q$ be a holomorphic self-map with no fixed points in $\B^q$, and let $\lambda_{f}\in (0,1]$ be the dilation of $f$ at its Denjoy--Wolff point. Then  $\lambda_{f}=e^{-c(f)}.$
\end{theorem}

\begin{definition}
Let $X$ be a complex manifold and let $f\colon X\to X$ be a holomorphic self-map.
A {\sl semi-model} for $f$ is a triple  $(\Lambda,h,\v)$,  where
$\Lambda$ is a complex manifold, $h\colon X\to \Lambda$ is a holomorphic mapping, and $\v\colon \Omega\to\Omega$ is an automorphism such that
$
h\circ f=\v\circ h$ and $\cup_{n\geq0}\v^{-n}(h(X))=\Lambda$.
We call the manifold $\Lambda$ the {\sl base space} and the mapping $h$ the {\sl intertwining mapping}.

Let $(Z,\ell,\tau)$ and   $(\Lambda, h, \v)$ be two semi-models for $f$. A {\sl morphism of semi-models}
$\hat\eta\colon (Z,\ell,\tau)\to (\Lambda, h, \v)$ is given by a  holomorphic map $\eta: Z\to \Lambda$
such that  the following diagram commutes:
\SelectTips{xy}{12}
\[ \xymatrix{X \ar[rrr]^h\ar[rrd]^\ell\ar[dd]^f &&& \Lambda \ar[dd]^\varphi\\
&& Z \ar[ru]^\eta \ar[dd]^(.25)\tau\\
X\ar'[rr]^h[rrr] \ar[rrd]^\ell &&& \Lambda\\
&& Z \ar[ru]^\eta.}
\]
\end{definition}
\begin{remark}\label{help}
It is shown in \cite[Lemmas 3.6 and 3.7]{ABmod} that if $ (Z,\ell,\tau), (\Lambda, h, \v)$ are semi-models for $f$, then there exists at most one morphism $\hat\eta\colon (Z,\ell,\tau)\to (\Lambda, h, \v)$, and that the holomorphic map $\eta: Z\to \Lambda$ is surjective.
\end{remark}

\begin{definition}
Let $X$ be a complex manifold and let $f\colon X\to X$ be a holomorphic self-map. Let $(Z, \ell,\tau)$ be a semi-model for $f$  whose base space $Z$ is Kobayashi hyperbolic. We say that  $(Z, \ell,\tau)$ is a {\sl canonical Kobayashi hyperbolic semi-model} for $f$  if for any semi-model
$(\Lambda, h,\varphi )$ for $f$ such that the base space $\Lambda$ is Kobayashi hyperbolic,  there exists a morphism of semi-models $\hat\eta\colon (Z, \ell,\tau)\to (\Lambda, h,\varphi )$.
 \end{definition}
 \begin{remark}\label{tango}
If $(Z, \ell,\tau)$ and $(\Lambda, h,\varphi )$ are two canonical Kobayashi hyperbolic semi-model for $f$, then they are isomorphic.
\end{remark}
In what follows we will need the following result from \cite{Amod} (see also \cite{ABmod}).
\begin{theorem}\label{principaleforward}
 Let $X$ be a cocompact  Kobayashi hyperbolic complex manifold, and let $f\colon X\to X$ be a holomorphic self-map. Then there exists a  canonical Kobayashi hyperbolic semi-model $(Z,\ell,\tau)$ for $f$, where $Z$ is a holomorphic retract of $X$.
Moreover, the following holds:
\begin{enumerate}
\item   for any $n\geq 0$,
$\lim_{m\to\infty}  (f^m)^*\, k_X=(\tau^{-n}\circ \ell)^* \, k_Z$;
\item the divergence rate of $\tau$ satisfies $c(\tau)=c(f).$
\end{enumerate}
\end{theorem}
\begin{remark}
Notice that $Z$ could reduce to a point. In such a case, $c(f)$ is necessarily zero.
\end{remark}

\section{A normal form for the  automorphisms of the polydisc}
We start by introducing some notation.
For $k\in\Natural$ and $j=1,\ldots,k$ we introduce notation for the following maps
\begin{align*}
\sigma_k&:\C^k\to\C^k; &&(z_1,z_2,\ldots,z_k)\mapsto(z_2,\ldots,z_{k-1},z_1),\\
\pi_{j,k}&:\C^k\to\C; &&z\mapsto \langle z,e_j\rangle,
\end{align*}
where $\langle\cdot,\cdot\rangle$ is the Hermitian inner product and $(e_j)$ is the standard basis in~$\C^k$.

The following classical theorem due to Poincar\'e describes all holomorphic automorphisms of~$\Delta^q$ in terms of disc automorphisms. We formulate this theorem for the ``poly-halfplane'' $\H^q$, which is, on the one hand, biholomorphic to the polydisc~$\Delta^q$, but, on the other hand, seems to be more convenient in the study of fixed-point free holomorphic self-maps.
\begin{theorem}[Poincar\'e]\label{TH_Poincare}
Every holomorphic automorphism $\tau$ of~$\H^q$ is a map of the form
$$
\H^q\ni(z_1,z_2,\ldots,z_q)\stackrel{\tau}{\mapsto}\big(\gamma_1(z_{p(1)}),\gamma_2(z_{p(2)}),\ldots,\gamma_q(z_{p(q)})\big),
$$
where $p$ is a permutation of~$\{1,\ldots,q\}$ and $\gamma_j\in\Aut(\H)$ for all~$j=1,\ldots, q$.
\end{theorem}
Unfortunately, the study of dynamics of polydisc automorphisms does not seem to benefit much from the direct application of the representation given in Theorem~\ref{TH_Poincare}.   In this section we show that Theorem~\ref{TH_Poincare} leads to another representation, a normal form for polydisc automorphisms, which turns out to be much more informative from the dynamical point of view. We start by introducing what we call \textsl{cycle automorphisms}.
\begin{definition}\label{DF_cycle-automorphism}
A \textsl{cycle automorphism} of~$\H^k$ is an automorphism~$\tau$ of the form
\begin{equation}\label{EQ_cycle-auto}
\H^k\ni z=(z_1,z_2,\ldots,z_k)\stackrel{\tau}{\mapsto}(\gamma_1(z_2),\gamma_2(z_3),\ldots,\gamma_k(z_1)),
\end{equation}
where $\gamma_j\in\Aut(\H)$ for all~$j=1,\ldots, k$.
\end{definition}
In this definition we allow $k=1$, in which case the cycle automorphisms are just automorphisms of~$\H$. The reason to consider cycle automorphisms is explained in the following remark.

\begin{remark}\label{RM_Poincare-thrm}
Using the decomposition of the permutation~$\sigma$ into cycles, one can rephrase Poincare's Theorem~\ref{TH_Poincare} as follows: {\it every holomorphic automorphism $\tau\colon\H^q\to\H^q$ can be represented, up to a permutation of the coordinates, as a direct sum of cycle automorphisms}. More precisely, for any $\tau\in\Aut(\H^q)$ there exists a partition ${\coprod_{\nu=1}^nJ_\nu=\{1,\ldots,q\}}$, bijective maps ${p_\nu\colon \{1,\ldots,k_\nu\}\to J_\nu}$ and cycle automorphisms ${\tau_\nu\colon \H^{k_\nu}\to\H^{k_\nu}}$, ${\nu=1,\ldots,n}$, such that
\begin{equation}\label{EQ_cycle-decomp}
\bar\pi_\nu\circ\tau=\tau_\nu\circ\bar\pi_\nu\quad\text{for all~$\nu=1,\ldots,n$},
\end{equation}
where $\bar\pi_\nu\colon \H^q\ni(z_1,\ldots,z_q)\mapsto(z_{p_\nu(1)},\ldots,z_{p_\nu(k_\nu)})\in\H^{k_\nu}$.
\end{remark}
In what follows, representation~\eqref{EQ_cycle-decomp} in the above remark will be referred to as a \textsl{cycle decomposition} of~$\tau$. Obviously, it is unique up to a change of order of $J_\nu$'s and \nv{pre-compositions} of $p_\nu$'s with cyclic shifts of the sets $\{1,\ldots,k_\nu\}$.

\begin{remark}\label{RM_kth-iterate-of-cycle-auto}
Note that for a cycle automorphism~$\tau$ defined by~\eqref{EQ_cycle-auto},
$$
\tau^{ k}(z_1,z_2,\ldots,z_k)=(\Gamma_1(z_1),\Gamma_2(z_2),\ldots,\Gamma_k(z_k)),
$$
where $\Gamma_j$, $j=1,\ldots,k$, are automorphisms of~$\H$ determined recurrently by $${\Gamma_1:=\gamma_1\circ\gamma_2\circ\ldots\circ\gamma_k}\quad\text{and}\quad {\Gamma_{j+1}=\gamma_j^{-1}\circ\Gamma_j\circ\gamma_j}~~\text{~for $j=1,\ldots,k-1$.}$$  In particular,  all $\Gamma_j$ are conjugate to each other.
\end{remark}

As a consequence, it is particularly easy to extend the dynamical classification of the disc automorphisms to the cycle automorphisms.

\begin{definition}\label{triaut}
Let $\tau\colon\H^k\to\H^k$  be a cycle automorphism. We say that $\tau$ is {\sl elliptic}, {\sl parabolic}, or {\sl hyperbolic}, depending on whether the automorphism $\Gamma_1$ (and hence any of~$\Gamma_j$'s) is an elliptic, parabolic or hyperbolic self-map of~$\H$, respectively.

In the elliptic case, we define the {\sl multipliers}~$\lambda_\tau$ of $\tau$  to be the $k$-th roots of $\Gamma_1'(p)$, where $p\in\H$ is the (unique) fixed point of~$\Gamma_1$. In the parabolic and hyperbolic cases, we define the {\sl dilation} of
$\tau$ to be $\lambda_\tau\coloneqq\sqrt[k] \lambda_{\Gamma_1}>0$, where $\lambda_{\Gamma_1}$ stands for the dilation of~$\Gamma_1$ at its Denjoy\,--\,Wolff point.
\end{definition}

The following result gives a normal form for cycle automorphisms (and, thanks to Remark \ref{RM_Poincare-thrm}, for general automorphisms). For completeness we will treat also the elliptic case, although  we will not need it in what follows.

\begin{theorem}\label{TH_canonical-from-cycle-auto}
Let $\tau:\H^k\to\H^k$ be a cycle automorphism. Then following statements hold.
\begin{itemize}
\item[(i)] If $\tau$ is hyperbolic or parabolic, then there exists $g\in\Aut(\H^k)$ of the form
$g(z_1,z_2,\ldots,z_k)=\big(g_1(z_1),g_2(z_2),\ldots,g_k(z_k)\big)$, where $g_j\in \Aut(\H)$, $\nv{j=1,\ldots,k}$, such that the following diagram commutes:
$$
\begin{CD}
\H^k @>\tau>> \H^k\\
@V{g}VV @VV{g}V\\
\H^k @>{L}>> \H^k
\end{CD}
$$
where $L:=\dfrac1{\lambda_{\mathrlap{\tau}}}\,\sigma_k$ in the hyperbolic case and $L:=\sigma_k\pm(1,1,\ldots,1)$ in the parabolic case.

\item[(ii)] If $\tau$ is elliptic, then for each of the $k$ values of the multiplier $\lambda_\tau$ there exists a biholomorphism ${g:\H^k\to\Delta^k}$ of the form\\
${g(z_1,z_2,\ldots,z_k)=\big(g_1(z_1),g_2(z_2),\ldots,g_k(z_k)\big)}$, where $g_j\colon \H\to \Delta$, $\nv{j=1,\ldots,k}$, are  biholomorphisms, such that the following diagram commutes:
$$
\begin{CD}
\H^k @>\tau>> \H^k\\
@V{g}VV @VV{g}V\\
\Delta^k @>{L}>> \Delta^k
\end{CD}
$$
where $L:=\lambda_\tau\sigma_k$.
\end{itemize}
\end{theorem}
\begin{proof}
Let us first assume first that $\tau$ is hyperbolic. Adopting notation of Remark~\ref{RM_kth-iterate-of-cycle-auto}, we see that there exists  $g_1\in \Aut(\H)$ such that $g_1\circ\Gamma_1=g_1/\lambda_{\Gamma_1}$.  Then the function $V_1:=g_1\circ\pi_{1,k}$ satisfies
\begin{equation}\label{EQ_Valiron-V_1}
V_1\circ\tau^{k}=\frac{1}{\lambda_{\mathrlap{\Gamma_1}}}\,V_1.
\end{equation}
For $j=2,\ldots,{k+1}$, define the functions $V_j$ recurrently by $V_j:=\lambda_{\tau}(V_{j-1}\circ\tau)$.
Thanks to~\eqref{EQ_Valiron-V_1} we see that  $V_{k+1}=V_1$. It follows that
\begin{equation}\label{EQ_V-vector-funct-eq}
(V_1,V_2,\ldots,V_k)\circ\tau=\frac{1}{\lambda_{\mathrlap{\tau}}}\,(V_2,\ldots,V_k,V_1)
\end{equation}
Using the very definition of cycle automorphism (Definition~\ref{DF_cycle-automorphism}), it is easy to see that $V_j=g_j\circ\pi_{j,k}$ for all $j=1,\ldots, k$, where the $g_j$'s are defined recurrently by $\nv{g_j:=\lambda_\tau (g_{j-1}\circ\gamma_{j-1})}$, $j=2,\ldots,k$. This proves the theorem for a hyperbolic cycle automorphisms~$\tau$.

The case of a parabolic automorphism~$\tau$ is treated  essentially in the same way, except that  $g_1\in \Aut(\H)$ is defined as a solution to the functional equation ${g_1\circ\Gamma_1=g_1\pm k}$, the equation \eqref{EQ_Valiron-V_1} is replaced by  $${V_1\circ\tau^{ k}=V_1\pm k},$$ and the functions~$V_j$, $\nv{g_j}$ for ${j=2,\ldots, k+1}$ are defined by $V_j:=(V_{j-1}\circ\tau)\mp1$ and $\nv{g_j:=(g_{j-1}\circ\gamma_{j-1})\mp1}$. In all the formulas, the choice of the sign ``$+$'' or ``$-$'' is determined uniquely by the canonical form of~$\Gamma_1$.

Finally, the proof for the elliptic case can be also carried out using essentially the same argument \nv{as above} if we define the biholomorphism $g_1\colon \H\to \Delta$ to be a solution to $g_1\circ\Gamma_1=\lambda_{\Gamma_1}g_1$, replace equation~\eqref{EQ_Valiron-V_1} with the Schr\"oder equation $$V_1\circ\tau^{k}=\lambda_{\Gamma_1}V_1,$$ and define the functions $V_j$, $\nv{g_j}$ for ${j=2,\ldots, k+1}$ recurrently by ${V_j:=\lambda^{-1}_{\tau}(V_{j-1}\circ\tau)}$ and $\nv{g_j:=\lambda_\tau^{-1} (g_{j-1}\circ\gamma_{j-1})}$.
\end{proof}

\begin{remark}
Using the fact that $L^{ k}$ is of the form $(z_1,\ldots,z_k)\mapsto(T(z_1),\ldots, T(z_k))$, where $T:\C\to\C$ is an affine map, it is not difficult to see that for the hyperbolic case, the intertwining map $g$ in the above theorem is unique up to multiplication by a positive number.
Similarly, for the parabolic case, $g$ is unique  up to post-composing with a translation of the form $z\mapsto z+(a,\ldots,a)$, where $a\in\Real$. The canonical form $L$ is unique in both cases, in particular, the choice of the sign for $L=\sigma_k\pm(1,\ldots,1)$ in the parabolic case is determined uniquely by the automorphism~$\tau$.
At the same time, for the elliptic case the canonical form~$L$ is determined by~$\tau$ \textit{and} by the choice of the multiplier~$\lambda_\tau$. Independently of the choice of the multiplier, the intertwining map $g$ is unique up to multiplication by a number of absolute value one, unless $\tau^{ k}=\id_{\H^k}$.  The canonical forms corresponding to different choices of the multiplier are conjugate by the linear map
$
(z_1,z_2,\ldots,z_k)\mapsto (z_1,\mu z_2,\ldots,\mu^{k-1}z_k),
$
where~${\mu^k=1}$. Finally, in the ``degenerate'' case $\tau^{k}=\id_{\H^k}$, the intertwining map~$g$ is defined up to an arbitrary holomorphic automorphism of~$\H$.
\end{remark}
\nv{
\begin{remark}
From Theorem~\ref{TH_canonical-from-cycle-auto} it follows that a cycle automorphism~$\tau\colon \H^k\to \H^k$ is elliptic if and only if it has a fixed point in~$\H^k$.
\end{remark}}

\section{Dynamical classification for holomorphic self-maps of the polydisc}

We now introduce a classification for holomorphic self-maps of the polydisc using the divergence rate.
\begin{definition}\label{definition}
Let $f\colon \H^q\to \H^q$ be a holomorphic self-map. Then we say that
\begin{itemize}
\item[i)] $f$ is {\sl elliptic} if it admits a fixed point $z\in \H^q$,
\item[ii)] $f$ is {\sl parabolic} if it admits no fixed point in $\H^q$ and  $c(f)=0$,
\item[iii)] $f$ is {\sl hyperbolic} if $c(f)>0$ (which implies that $f$ admits no fixed point in $\H^q$).
\end{itemize}
If $f$ is not elliptic, then we define its {\sl dilation} to be $\lambda_{f}\coloneqq e^{-c(f)}\in (0,1]$.
 We say that
 $f$ is of {\sl non-zero step}  if $\lim_{n\to\infty}k_{\H^q}(f^n(z),f^{n+1}(z))\neq 0$ for all~$z\in \H^q$.
\end{definition}

It is easy to see that, if $\tau$ is a cycle automorphism, then this definition is coherent with Definition~\ref{triaut}.
We now study the case of an arbitrary automorphism $\tau$ of $\H^q$.

\begin{proposition}\label{PR_multiplier-and-div-rate}
Let $\tau\in\Aut(\H^q)$. Then
$
\SD\tau=\max_{\nu} c(\tau_\nu),
$
where the maximum is taken over all cycle automorphisms~$\tau_\nu$ in the cycle  decomposition of $\tau$.
\end{proposition}
\begin{proof}
We adopt here the notation of Remark~\ref{RM_Poincare-thrm}.
Using that remark together with Remark~\ref{RM_kth-iterate-of-cycle-auto}, we see that if $Q:=\prod_{\nu=1}^n k_\nu$, then
$$
\tau^{ Q}(z_1,z_2,\ldots,z_q)=\big(T_1(z_1),T_2(z_2),\ldots,T_q(z_q)\big)\quad \text{for all~$(z_1,\ldots,z_q)\in\H^q$,}
$$
where $T_j\in\Aut(\H)$, $j=1,\ldots,q$. Moreover, $c(T_j)=Qc(\tau_\nu)$ for all~$j\in J_\nu$ and all ${\nu=1,\ldots,n}$.

Recall that $ k_{\H^q}(z,w)=\max_{j=1,\ldots,q} k_{\H}(z_j,w_j)$ for any pair of points ${z=(z_1,\ldots,z_q)}$ and  ${w=(w_1,\ldots,w_q)}$ in~${\H^q}$. Then  $\SD{\tau^{ Q}}=\max_\nu  Qc(\tau_\nu)$, which implies the result, see~\cite[Remark~2.5]{ABmod}.
\end{proof}

This immediately yields the following
\begin{corollary}\label{peraut}
Let $\tau\in\Aut(\H^q)$. Then:
\begin{enumerate}
\item $\tau$ is elliptic if and only if its cycles are all elliptic,
\item $ \tau $ is parabolic if and only if it admits a parabolic cycle and does not admit hyperbolic ones,
\item $\tau$ is hyperbolic if and only if it admits a hyperbolic cycle, and in this case its dilation $\lambda_\tau$ is the minimum of the dilations of its hyperbolic cycles.
\end{enumerate}
\end{corollary}

\section{The Valiron equation}\label{S_Valiron}
First we show that  hyperbolicity of $f$ is a necessary condition for the Valiron equation to have solutions with values in~$\H$.
\begin{proposition}
Let $f\colon \H^N\to\H^N$  be a holomorphic self-map. Suppose that there exists a holomorphic function $V\colon \H^N\to \H$ and a constant $\mu\in (0,1)$  such that
$$V\circ f=\frac{1}{\mu}V.$$
Then $f$ is hyperbolic with dilation $\lambda_{f}\leq \mu$.
\end{proposition}
\begin{proof}
Since the divergence rate of the automorphism $z\mapsto \frac{1}{\mu}z $ equals $-\log \mu$, by \cite[Lemma~2.9]{ABmod} we have that
$-\log\lambda_{f}\coloneqq c(f)\geq -\log \mu>0$.
\end{proof}

\begin{definition}
Let $f:\H^N\to\H^N$ be a hyperbolic holomorphic self-map with dilation~$\lambda_{f}$.
By a \textsl{Valiron function $V$} of $f$ we mean
a holomorphic function $V:\H^N\to\H$ that solves the {\sl Valiron equation}
\begin{equation}\label{valironequation}
V\circ f=\frac{1}{\lambda_{f}}V.
\end{equation}
\end{definition}

For hyperbolic cycle automorphisms we can give a complete characterization of the family of all Valiron functions. According to Theorem \ref{TH_canonical-from-cycle-auto}, without loss of generality we may assume that $\tau$ is in its canonical form, i.e. $\tau\coloneqq\frac{1}{\lambda}\sigma_k$ for some constant~${\lambda\in(0,1)}$.
\begin{proposition}\label{PR_Valiron-functions-for-cycle-autos}
Let $\tau\coloneqq\frac{1}{\lambda}\sigma_k$, where $\lambda\in(0,1)$. Then $V$ is a Valiron function of~$\tau$ if and only if
\begin{itemize}
\item[(i)] $V(rw)=rV(w)$ for all $r>0$ and all~$w\in\H^k$;
\item[(ii)] $V\circ\sigma_k=V$.
\end{itemize}
Moreover, $V(\H^k)=\H$ for any Valiron function $V$ of $\tau$.
\end{proposition}
\begin{proof}
It is an easy exercise to check that if $V:\H^k\to\H$ is holomorphic and satisfies (i) and~(ii), then it is a Valiron function of~$\tau$.

Let us prove the converse. Suppose that $V$ is a Valiron function of~$\tau$.  Then
\begin{equation}\label{EQ_G-property1a}
V(w/\lambda^k)=V(w)/\lambda^k \quad\text{for all~} w\in\H^k.
\end{equation}
We have to show that it satisfies (i) and~(ii).
For $a=(a_1,\ldots,a_k)\in(\R^+)^k$ consider the map $V_a(\zeta):=V(\zeta a)$, $\zeta\in\H$.
Then $$V_a(\zeta/\lambda^k)=V_a(\zeta)/\lambda^k\quad \text{for all~} \zeta\in \H.$$
By a result of Heins \cite{heins}, it follows that $V_a(\zeta)=-iV(ia)\zeta$ for all $\zeta\in\H$.

In particular,  $$V(ira)=V_a(ir)=V(ia)r \quad  \text{for all}\  r>0~\text{and all}~a\in(\R^+)^k.$$ Applying now the Uniqueness Principle for holomorphic functions, we obtain~(i). In its turn (i) yields $$V\circ \sigma_k=V\circ (\lambda\tau)=\lambda( V\circ\tau)=V.$$

Finally, the chain of inclusions $$\H=V_a(\H)\subset V(\H^k)\subset\H$$ shows that $V(\H^k)=\H$ whenever $V$ is a Valiron function of~$\tau$.
\end{proof}

Now we consider the case of an arbitrary hyperbolic automorphism $\tau:\H^q\to\H^q$ with dilation $\lambda$. According to Remark~\ref{RM_Poincare-thrm} and Proposition~\ref{PR_multiplier-and-div-rate}, up to a reordering of variables, $\tau$ can be written in the following form
\begin{equation}\label{EQ_hyper_aut-splitting}
\H^q\cong\H^m\times\H^{q-m}\ni(z,w)\mapsto (\hat\tau(z),\tilde\tau(w))\in \H^m\times\H^{q-m},
\end{equation}
where $0<m\le q$ and the cycle decomposition  for~$\hat\tau$ contains only hyperbolic cycle automorphism with dilation $\lambda$ while the cycle decomposition for~$\tilde\tau$ may contain only non-hyperbolic cycle automorphism or hyperbolic cycle automorphisms with strictly greater dilation.

Applying Theorem~\ref{TH_canonical-from-cycle-auto} to each of the cycle automorphisms in the decomposition of~$\hat\tau$, we see that we can assume without loss of generality that $\hat\tau\coloneqq \frac{1}{\lambda}\hat\sigma,$ where
$\hat\sigma:\C^m\to\C^m$ is a linear map of the form $$\hat\sigma(z_1,z_2,\ldots,z_m):=(z_{\hat p(1)},z_{\hat p(2)},\ldots,z_{\hat p(m)})$$ defined by  a permutation $\hat p$  of~$\{1,\ldots,m\}$.
\begin{theorem}\label{TH_Valiron-functions-for-autos}
In the above notation, a holomorphic function $V:\H^m\times\H^{q-m}\to\H$ is a Valiron function of~$\tau$ if and only if $V(z,w)$  does not depend on~$w\in \H^{q-m}$ and
\begin{itemize}
\item[(i)] $V(rz)=rV(z)$ for all $r>0$ and all~$z\in\H^m$;
\item[(ii)] $V\circ\hat\sigma=V$.
\end{itemize}
Moreover, $V(\H^q)=\H$ for any Valiron function $V$ of the automorphism~$\tau$.
\end{theorem}
\begin{remark}\label{RM_inf_Valiron-functions}
It follows immediately from the above theorem that a hyperbolic automorphism of a polydisc has infinitely many Valiron functions.
\end{remark}
\begin{proof}[Proof of Theorem~\ref{TH_Valiron-functions-for-autos}]
As in Proposition~\ref{PR_Valiron-functions-for-cycle-autos}, it is easy to check that  if $V(z,w)$ does not depend on ${w\in \H^{q-m}}$ and satisfies (i) and~(ii), then it is a Valiron function of~$\tau$.

Now we show that if $V$ is a Valiron function, then $V(z,w)$ does not depend on~$w\in \H^{q-m}$.
Once this is proved, the rest of the proof of the theorem repeats almost literally that of~Proposition~\ref{PR_Valiron-functions-for-cycle-autos} except that the shift map~$\sigma_k$ is replaced by the map~$\hat\sigma$.

We assume $m<q$, otherwise there is nothing to prove.
Let $Q\in\Natural$ be defined as in the proof of Proposition~\ref{PR_multiplier-and-div-rate}. Then we have:
\begin{itemize}
\item[(a)] $\hat\tau^{Q}(z)=z/\lambda^Q$ for all~$z\in\H^m$;
\item[(b)] in particular, $ k_{\H^m}(z,\hat\tau^{ jQ}(z))=j k_\H(i,i/\lambda^Q)=jQ \SD\tau$ for all $j\in\Natural$ and any $z\in (i\R^+)^m$;
\item[(c)] $ k_{\H^{q-m}}(w,\tilde\tau^{ jQ}(w))/j\to Q\SD{\tilde\tau}<Q\SD{\tau}$ as~$j\to+\infty$  uniformly on compact subsets with respect to~$w\in \H^{q-m}$.
\end{itemize}

Fix some $z_0\in (i\R^+)^m$ and some closed ball $B\subset\H^{q-m}$.  According to (b) and~(c) there exists $j_0\in\Natural$ such that if~$M:=j_0Q$, then $$ k_{\H^q}\big((z_0,w), \tau^{ M}(z_0,w)\big)=\max\big\{ k_{\H^m}\big(z_0, \hat\tau^{ M}(z_0)\big), k_{\H^{q-m}}\big(w, \tilde\tau^{ M}(w)\big)\big\}=M \SD\tau$$ for all~$w\in B$. Therefore, on the one hand,
$$
 k_\H\big(V(z_0,w),V(\tau^{M}(z_0,w))\big)\le k_{\H^q}\big((z_0,w),\tau^{M}(z_0,w)\big)=M \SD\tau.
$$
On the other hand, $V(\tau^{M}(z_0,w))=V(z_0,w)/\lambda^M$ because $V$ is a Valiron function of~$\tau$. Thus
$$
 k_\H\big(V(z_0,w),V(\tau^{M}(z_0,w))\big)\ge  k_{\H}(i,i/\lambda^M)=M \SD\tau,
$$
and   the equality is attained only if $V(z_0,w)\in i\R^+$.

This shows that $V(z_0,w)\in i\R^+$ for all $w\in B$ and hence, by the Open Mapping Theorem,  $V(z_0,\cdot):\H^{q-m}\to\H$ is a constant function for any~$z_0\in  (i\R^+)^m$. By the Uniqueness Principle for holomorphic functions, this means that $V(z,w)$ depends only on~$z$ in the whole~$\H^m\times\H^{q-m}$.
\end{proof}

\nv{Now we prove the main result of this section.}
\begin{theorem}\label{TH_Valiron-equation}
 Let $f\colon \H^N\to \H^N$ be a hyperbolic holomorphic self-map with dilation $\lambda_{f}$.
 Then $f$ admits a  canonical Kobayashi hyperbolic semi-model $(\H^q, \ell,\tau)$, where $1\leq q\leq N$ and  $\tau$ is a hyperbolic automorphism of $\H^N$ with dilation $\lambda_{\tau}=\lambda_f$.
Moreover, a holomorphic function $V\colon \H^N\to \H$ is a Valiron function of $f$ if and only if $V=\tilde V\circ \ell$, where  $\tilde V\colon \H^q\to\H$ is a Valiron function of the automorphism $\tau$. In particular, every Valiron  function $V$ of $f$ satisfies
\begin{equation}\label{prinh}
\bigcup_{n\geq 0} \lambda_{f}^{n}\,V(\H^N)=\H.
\end{equation}
\end{theorem}
\begin{proof}
By Theorem \ref{principaleforward}, there exists  a canonical Kobayashi hyperbolic semi-model $(Z,\hat\ell,\hat\tau)$  for $f$.
Since $Z$ is a holomorphic retract of $\H^N$, by \cite{HS} there exists a biholomorphism $\psi\colon Z\to\H^q$, where $0\leq q\leq N$.
Clearly $$(\H^q,\ell\coloneqq \psi\circ\hat\ell,\tau\coloneqq \psi\circ\hat\tau\circ \psi^{-1})$$ is also a  canonical Kobayashi hyperbolic semi-model for $f$.

By assertion~(2) of  Theorem \ref{principaleforward} we have that $c(\tau)=c(f)>0$. Hence $q\geq 1$ and $\tau$ is a hyperbolic
automorphism with $\lambda_\tau=\lambda_f$.

If $\tilde V\colon \H^q\to\H$ is a Valiron function of the automorphism $\tau$, then clearly $V\coloneqq \tilde V\circ \ell$ is a Valiron function of $f$.
Conversely, assume that  $V\colon \H^N\to \H$ is a Valiron function of $f$, and set $\Omega\coloneqq \cup_{n\geq 0}\lambda_{f}^n\,V(\H^N)$.
Then $(\Omega, V,z\mapsto\frac{1}{\lambda_{f}}z)$ is a semi-model for $f$ with Kobayashi hyperbolic base space. Since $(\H^q,\ell,\tau)$ is a canonical Kobayashi  hyperbolic semi-model for~$f$, there exists a semi-model morphism
$\hat {\eta}\colon (\H^q,\ell,\tau)\to(\Omega, V,z\mapsto\frac{1}{\lambda_{f}}z)$. The holomorphic function $\tilde V\coloneqq \eta$ is a Valiron function of  $\tau$ and it satisfies $V=\tilde V\circ \ell$.

\nv{Finally, equality~\eqref{prinh} holds because, by Theorem~\ref{TH_Valiron-functions-for-autos},  $\nv{\H=\tilde V(\H^q)\subset\Omega\subset\H.}$}
\end{proof}

\nv{
\begin{remark}
At the end of preparation of this paper, the authors got to know that Wang and Deng~\cite{WD} recently proved existence of a Valiron function under quite restrictive additional assumptions: in particular, it is supposed that there exists an orbit $(f^n(z_0))_{n\in\Natural}$ converging to a point on the boundary of the polydisc within a $K$-region. Consider, e.g., the self-map $\H\times\D\ni(z,w)\mapsto \big(e^{\alpha\pi}z,\tfrac{2+w}3\exp(i\log z)\big)\in\H\times\D$, where $\alpha>0$ is an irrational number and $\log z$ stands for the single-valued branch that takes values in~$\{\zeta\colon \Im\zeta\in(0,\pi)\}$ for $z\in\H$. It is easy to see that according to Definition~\ref{definition}, this self-map is hyperbolic and hence, by Theorem~\ref{TH_Valiron-equation}, it admits a Valiron function, although neither the map itself nor any of its iterates has an orbit convergent to a boundary point.
\end{remark}
}

\section{The Abel equation}\label{S_Abel}
First we show that  $f$ must be a self-map of non-zero step for the Abel equation to have a solution with values in~$\H$.
\begin{proposition}
Let $f\colon \H^N\to\H^N$  be a holomorphic self-map, and assume that there exists a holomorphic function $\Theta\colon \H^N\to \H$ and a constant $\alpha\in \{-1,1\}$  such that
$$\Theta\circ f=\Theta+ \alpha.$$
Then $f$ is  of non-zero step.
\end{proposition}

\begin{proof}
For all $x\in \H^N$ and all~$m\ge0$,  we have
 \begin{multline*}
 k_{\H^N}(f^m(x),f^{m+1}(x))\ge k_{\H}\big(\Theta(f^m(x)),\Theta(f^{m+1}(x))\big)\\=k_{\H}\big(\Theta(x)+\alpha m,\Theta(x)+\alpha (m+1)\big)=k_{\H}\big(\Theta(x),\Theta(x)+\alpha\big).
\end{multline*}
Therefore,
$$s(x)=\lim_{m\to\infty} k_{\H^N}(f^m(x),f^{m+1}(x))\ge k_{\H}\big(\Theta(x),\Theta(x)+\alpha\big)>0$$
for all $x\in \H^N$, what was to be shown.
\end{proof}

\begin{theorem}\label{par}
 Let $f\colon \H^N\to \H^N$ be a parabolic holomorphic self-map of non-zero step.
 Then it admits a  canonical Kobayashi hyperbolic semi-model  $(\H^q, \ell,\tau)$, where $1\leq q\leq N$ and $\tau$ is a parabolic automorphism of $\H^N$.
 Moreover, there exists $\alpha\in \{-1,1\}$ and a  holomorphic solution $\Theta\colon \H^N\to \H$ to the Abel equation
 $$\Theta\circ f=\Theta +\alpha$$
that satisfies
 \begin{equation}\label{prin}
 \bigcup_{n\geq 0}[\Theta(\H^N)-n\alpha]=\H.
 \end{equation}
 \end{theorem}
\begin{proof}
By Theorem \ref{principaleforward} there exists  a canonical Kobayashi hyperbolic semi-model $(Z,\hat\ell,\hat\tau)$  for $f$.
Since $Z$ is a holomorphic retract of $\H^N$, by \cite{HS} there exists a biholomorphism $\psi\colon Z\to\H^q$, where $0\leq q\leq N$.
Clearly $$(\H^q,\ell\coloneqq \psi\circ\hat\ell,\tau\coloneqq \psi\circ\hat\tau\circ \psi^{-1})$$ is also a  canonical Kobayashi hyperbolic semi-model for $f$.

Fix any $z\in \H^q$. By the definition of a semi-model, there exists $x\in \H^N$ and $n\geq 0$ such that $\tau^{-n}(\ell(x))=z$.
Then, by assertion~(1) of Theorem~\ref{principaleforward}, $$k_{\H^q}(z,\tau(z))=s(x)>0.$$
Hence $q\geq 1$ and $\tau$ is not elliptic. By (2) of Theorem \ref{principaleforward}, we  have  $c(\tau)=c(f)=0$. Hence $\tau$ is parabolic. By Corollary \ref{peraut}, the cycle decomposition of~$\tau$ contains at least one parabolic cycle automorphism~$\tau_\nu$. By Theorem~\ref{TH_canonical-from-cycle-auto}\,(i), there exists $g\in\Aut(\H^{k_\nu})$ such that
\begin{equation}\label{EQ_tau-nu}
g^{-1}\circ\tau_\nu\circ g=\sigma_{k_\nu}+\alpha(1,1,\ldots,1)
\end{equation}
for $\alpha\equiv1$ or $\alpha\equiv-1$.

With notations of Remark~\ref{RM_Poincare-thrm},
set $\mu(z_1,\ldots,z_{k_\nu})\coloneqq\frac{1}{k_\nu}\sum_{j=1}^{k_\nu}z_j$ and $A\coloneqq \mu\circ g\circ\bar\pi_\nu$.
Using \eqref{EQ_cycle-decomp} and~\eqref{EQ_tau-nu}, it easy to see that  $A\circ\tau=A+\alpha$. Therefore,
$\Theta\circ f=\Theta+\alpha$,
 where $\Theta:=A\circ\ell$. This proves the existence of solutions to the Abel equation.

Finally, note that $A(\H^q)=\mu(\H^{k_\nu})=\H$.  By definition of a semi-model, we have $\cup_{n\ge0}\tau^{-n}(\ell(\H^N))=\H^q$. Therefore, $$\H=A(\H^q)=\bigcup_{n\ge0}A(\tau^{-n}(\ell(\H^N)))=\bigcup_{n\ge0}[A(\ell(\H^N))-\alpha n].$$ This implies~\eqref{prin}. The proof is complete.
\end{proof}
 \begin{remark}
 There are parabolic univalent self-maps of the polydisc whose canonical Kobayashi hyperbolic semi-models are elliptic\footnote{A semi-model  $(Z,\ell,\psi)$ is said to be elliptic, if the automorphism $\psi$ has a fixed point in~$Z$.}.  An elementary example is the self-map of $\Delta \times \Ha$ defined by $f(z,w)=(\lambda z, w+i)$ with $|\lambda|=1$. Its canonical Kobayashi hyperbolic semi-model is  $(\Delta,(z,w)\mapsto z,z\mapsto \lambda z)$. Clearly, by Theorem \ref{par} such self-maps are not of non-zero step. Whether or not such a phenomenon can appear in the unit ball is an open question \cite[Section 5.5]{ABmod}.
 \end{remark}

\end{document}